\documentclass{amsbook}

\usepackage{amsmath}
\usepackage[T1]{fontenc}
\usepackage{ae,aecompl}

\usepackage{amssymb,amscd}
\usepackage{url}
\usepackage{amsthm,amsfonts,amsxtra}
\usepackage{graphicx}
\usepackage{bbm}

\numberwithin{equation}{section}
\numberwithin{table}{section}
\numberwithin{figure}{section}
\newtheoremstyle{bold}
{.5\baselineskip}{.5\baselineskip}{\itshape}{}{\bfseries}{.}{.5em}{}
\newtheoremstyle{shy}
{.5\baselineskip}{.5\baselineskip}{}{}{\bfseries}{.}{.5em}{}

\makeatletter
\def\@captionfont{\small}
\newenvironment{beweis}[1][\proofname]{\par
  \pushQED{\qed}%
  \normalfont \topsep6\p@\@plus6\p@\relax
  \trivlist
  \item[\hskip\labelsep
        \scshape
    #1\@addpunct{.}]\ignorespaces
}{%
  \popQED\endtrivlist\@endpefalse
}

\def\mychapter{%
  \if@openright\cleardoublepage\else\clearpage\fi
 \thispagestyle{empty}\global\@topnum\z@
  \@afterindenttrue \secdef\@mychapter\@schapter}

\def\@mychapter[#1]#2#3{\refstepcounter{chapter}%
  \ifnum\c@secnumdepth<\z@ \let\@secnumber\@empty
  \else \let\@secnumber\thechapter \fi
  \typeout{\chaptername\space\@secnumber}%
  \def\@toclevel{0}%
  \ifx\chaptername\appendixname \@tocwriteb\tocappendix{chapter}{#2\\ \scshape #3}%
  \else \@tocwriteb\tocchapter{chapter}{#2\\ \scshape #3}\fi
  \chaptermark{#1}%
  \addtocontents{lof}{\protect\addvspace{10\p@}}%
  \addtocontents{lot}{\protect\addvspace{10\p@}}%
  \@mymakechapterhead{#2}{#3}\@afterheading}

\def\@mymakechapterhead#1#2{\global\topskip 7.5pc\relax
  \begingroup
  \fontsize{\@xivpt}{18}\bfseries\centering
    \ifnum\c@secnumdepth>\m@ne
      \leavevmode \hskip-\leftskip
      \rlap{\vbox to\z@{\vss
          \centerline{\normalsize\mdseries
              \uppercase\@xp{ }}
          \vskip 3pc}}\hskip\leftskip\fi
     #1\par \vskip 1pc
     \Large\mdseries\scshape\centering
     #2\par \endgroup
  \skip@34\p@ \advance\skip@-\normalbaselineskip
  \vskip\skip@ }

\def\section{\@startsection{section}{1}%
  \z@{.9\linespacing\@plus\linespacing}{.5\linespacing}%
  {\large\bfseries\boldmath\centering}}

\def\subsection{\@startsection{subsection}{2}%
  \z@{.7\linespacing\@plus\linespacing}{.5\linespacing}%
  {\normalfont\scshape\centering}}

\def\theindex{\@restonecoltrue\if@twocolumn\@restonecolfalse\fi
  \columnseprule\z@ \columnsep 35\p@
  \@indextitlestyle
  \thispagestyle{empty}%
  \let\item\@idxitem
  \parindent\z@  \parskip\z@\@plus.3\p@\relax
  \raggedright
  \hyphenpenalty\@M
  \footnotesize}

\renewcommand{\@bibtitlestyle}{%
  \@xp\section\@xp*\@xp{\bibname}%
}
\renewcommand{\tocchapter}[3]{%
  \indentlabel{\@ifnotempty{#2}{\ignorespaces#1 #2.\quad}}#3}

\renewcommand{\tocsection}[3]{%
  \indentlabel{\@ifnotempty{#2}{\makebox[3.2em][l]{\ignorespaces#1 #2.}}}#3}

\makeatother

\renewcommand{\tocappendix}[3]{%
  \indentlabel{#1.\quad}#3}
\renewcommand{\tocappendix}[3]{%
  \indentlabel{\makebox[5.7em][l]{\ignorespaces#1.}}#3}

\renewcommand{\bibname}{References}

\renewcommand{\geq}{\geqslant}
\renewcommand{\le}{\leqslant}
\renewcommand{\leq}{\leqslant}

\theoremstyle{bold}
\newtheorem{theorem}{Theorem}[section]
\newtheorem{proposition}[theorem]{Proposition}

\theoremstyle{shy}

\newcommand{\cM}{\mathcal{M}}

\newcommand{\cR}{\mathcal{R}}

\newcommand{\one}{\mathbbm{1}}

\newcommand{\NN}{\mathbb{N}}

\newcommand{\RR}{\mathbb{R}}

\newcommand{\dd}{\ts\mathrm{d}\ts}
\newcommand{\ee}{\ts\mathrm{e}\ts}

\newcommand{\ts}{\hspace{0.5pt}}

\usepackage{xcolor}

\makeindex

\begin{document}

%\frontmatter
%\title{Probabilistic Structures in Evolution}

%\author{Ellen Baake \and Anton Wakolbinger \and (Editors)}

%\maketitle

%\tableofcontents
%\include{preface}        % Preface
%\include{contributors}    % List of Contributors
%\include{intro}       % Foreword

\mainmatter
\mychapter{Multiple-merger genealogies --- models, consequences, inference}{Fabian~Freund}\label{chap:FF}

Trees corresponding to $\Lambda$- and $\Xi$-$n$-coalescents can be both
 quite similar and fundamentally different compared to bifurcating
tree models based on Kingman's $n$-coalescent. This has consequences
for inference of a well-fitting gene genealogy as well as for
assessing biological properties of species having such sample
genealogies.
Here, mathematical properties concerning clade sizes in the tree as
well as changes of the tree when the samples are enlargened are
highlighted. To be used as realistic genealogy models for real
populations, an extension for changing population sizes is discussed.\footnote{This is a review of my project within the DFG Priority Programme 1590 "Probabilistic Structures in Evolution". It will appear in "Probabilistic Structures in Evolution", ed. by E. Baake and A. Wakolbinger, EMS Publishing House, Zurich.}

\section[Multiple-merger coalescents]{Multiple-merger
  coalescents}\label{FF-sect:intro}
Two decades ago, $\Lambda$- and shortly later the bigger class of
$\Xi$-$n$-coalescents \index{coalescent!Xi@$\Xi$-$n$-}
\index{coalescent!Lambda@$\Lambda$-$n$-} have been introduced and become an
active area of mathematical research ever since. These are Markov
processes $\Pi^{(n)}=(\Pi^{(n)}_t)_{t\geq 0}$ on the set of partitions
of $\{1,\ldots,n\}$. Transitions are mergers of partition blocks. A
$\Xi$-$n$-coalescent's rates are characterised by a finite measure
$\Xi$ with positive mass on the simplex
$$
\Delta=\{(x_i)_{i\in \NN}:x_1\geq x_2\geq \ldots, \sum_{i\in\NN}
x_i\leq 1\}.
$$
An intuitive way of describing the transitions is via a
Poisson construction, see \cite{FF-Pitman1999, FF-Schweinsberg2000}\index{paintbox construction}: Transitions
are possible at times $t\in[0,\infty)$ that feature a point $(t,x)$ of
a Poisson point process\index{Poisson!point process}\index{process!Poisson!point} on $[0,\infty)\times \Delta$ with intensity
$\dd t\otimes \frac{\Xi^*(\dd x)}{(x,x)}$, where $\Xi^*$ is
the restriction of $\Xi$ to $\Delta\setminus\{(0,0,\ldots)\}$ and
$(x,x):=\sum_{i\in\NN} x^2_i$.  At
time $t$, consider the 'paintbox' $x=(x_1,x_2,\ldots)$: Each partition
block $i$ present at $t$ draws independently a colour $C_i$, which
paints it with colour $j\in\mathbb{N}$ with $P(C_i=j)=x_j$ and makes it
colourless with $P(C_i=0)=1-\sum_{m\in\NN}x_m$.  Merge each set of
blocks of the same colour. If $\Xi$ has positive mass on the origin,
say $a=\Xi(\{(0,0,\ldots)\})$, also allow additional transitions for
each pair of blocks present with transition rate $a$.

Such coalescents allow transitions that are mergers of sets of any
number of blocks present into multiple new blocks. If $\Xi$ has only
positive mass on the first coordinate of the simplex, only a single
set of blocks can be merged into one new block as a transition. This
subclass of $\Xi$-$n$-coalescents are called
$\Lambda$-$n$-coalescents, $\Lambda$ being the restriction of $\Xi$ to 
the first dimension of $\Delta$, a finite measure on $[0,1]$.

For any $\Xi$, $\Pi^{(n)}$ almost surely reaches its absorbing state
$\{1,\ldots,n\}$ in finite time. For any fixed time $t$, the partition
$\Pi^{(n)}_t$ of the $\Xi$-$n$-coalescent is exchangeable,
i.e. $\sigma(\Pi^{(n)}_t)\stackrel{d}{=}\Pi^{(n)}_t$ for any
permutation $\sigma$ of $\{1,\ldots,n\}$.  This extends to certain
non-deterministic times, for instance the partition before absorption
in $\{\{1,\ldots,n\}\}$ is exchangeable. Moreover, there exists a
Kolmogorov extension $(\Pi_t)_{t\geq 0}$ on the set of partitions of
$\NN$ so that its restriction to $[n]:=\{1,\ldots,n\}$ is a
$\Xi$-$n$-coalescent for any $n\in\NN$, see \cite[Theorem
2]{FF-Schweinsberg2000}. $\Pi=(\Pi_t)_{t\geq 0}$ is called a
$\Xi$-coalescent.  Exchangeability extends to the $\Xi$-coalescent in
the sense that any permutation of a finite set $S\subset \NN$ does not
change the distribution of $\Pi$.  A consequence of this is that the
limit frequencies $\lim_{n\to\infty} B^{(n)}_i(t)/n$ of the
partition block sizes $B^{(n)}_i(t)$ of the restrictions of $\Pi_t$ to
$[n]$ (e.g. ordered by sizes) exist almost surely, this is called
Kingman's correspondence (which is essentially de Finetti's theorem,
see e.g.
\cite[Appendix]{FF-Schweinsberg2000}). 
$\Xi$-$n$-coalescents and related models also appear as central
objects in other chapters, see the contributions of Birkner
and Blath~\cite{FF-MBJB20}, of Kersting and Wakolbinger~\cite{FF-GKAW20}
as well as of Sturm~\cite{FF-AS20} in this volume.

The Poisson construction also shows fundamental differences between
different choices of $\Xi$ (respectively $\Lambda$). While choosing
$\Lambda=\delta_0$ leads to Kingman's $n$-coalescent, which only allows
binary mergers, any $\Xi$ with mass outside of $(0,0,\ldots)$ allows
for mergers of more than two blocks (which are the name-giving
multiple mergers).  If $\frac{\Xi(\dd x)}{(x,x)}$ is a
finite measure, the Poisson point process almost surely only has
finitely many points on any set $[0,t]\times\Delta$. For such $\Xi$,
the corresponding subclass of coalescent processes are called simple
$\Xi$-($n$-)coalescents.  If
$\sum_{i\in\NN}x_i\frac{\Xi(\dd x)}{(x,x)}$ %/{\sum_{i\in\NN}x^2_i}$
is a finite
measure, at least the number of mergers the block containing a fixed
$i\in\mathbb{N}$ can participate in is almost surely
finite. Corresponding $\Xi$(-$n$)-coalescents are called coalescents
with dust.\index{dust} All simple $\Xi$(-$n$)-coalescents are also coalescents
with dust.  A further property of $\Xi$-coalescents is whether they
stay infinite, i.e. which have infinitely many blocks almost surely at
any time, or whether they come down from infinity,\index{coming down from infinity} i.e. they have
finitely many blocks at any time $t>0$ almost surely. All
$\Xi$-coalescents with dust that do not allow for Poisson points where
all blocks are coloured with finitely many colours, i.e. that have
$\Xi(\{x\in\Delta:\sum_{i=1}^k x_i=1 \mbox{ for a }k\in\NN\})=0$ stay
infinite, see e.g. \cite[Sect. 3]{FF-freund2017size}. The case of
$\Xi$-coalescents without dust is more complex, see
e.g.~\cite{FF-herriger2012conditions}.  We will need later that the times
to absorption of $\Xi$-$n$-coalescents not staying infinite are
converging almost surely to a finite limiting variable for
$n\to\infty$: this follows from \cite[Lemma 31]{FF-Schweinsberg2000}
if the coalescent comes down from infinity. Otherwise, the same lemma shows
that $\Xi$ has positive mass on Poisson points that colour all present
blocks with finitely many colours. Thus, there is a finite waiting time
for such points, which then forces the number of blocks to be finite
and subsequently these will merge into absorption after a finite time
almost surely.

One can even put this stronger for both cases: Since only finitely
many blocks merge, exchangeability ensures that given $n$ large enough,
all the blocks of the $\Xi$-coalescent merging at the last collision
have at least one individual in the blocks of the
$\Xi$-$n$-coalescent. Thus, the height of the $\Xi$-$n$-coalescent is
the same as for the $\Xi$-coalescent above a certain (path-dependent)
$n_0$, if the $\Xi$-coalescent does not stay infinite.

Any $\Xi$-$n$-coalescent also encodes a random ultrametric\index{ultrametric}
(labelled) tree\index{tree!ultrametric} with $n$ leaves, where
ultrametric means that the path lengths
from leaves to root are identical for all
leaves.
We build this tree from leaves to root: Start with $n$ edges (initially
length 0) at the leaves corresponding to the blocks
$\{1\},\ldots,\{n\}$ at time 0. Elongate all edges by $l_1$, where
$l_1$ is the waiting time for the first transition of
$\Pi^{(n)}$. Then, the edges corresponding to the sets of merged
partition blocks of $\Pi^{(n)}$ at this transition are joined in new
internal nodes (for $\Xi$-$n$-coalescents, a transition may consist of
multiple simultaneous mergers).
%The labels of each or these new internal nodes are the integers from all blocks that were merged in the merger this node corresponds to . 
Start a new edge at each newly introduced node (length 0) and elongate
again all branches not yet connected to two nodes for $l_2$, the
waiting time until the next transition of the $n$-coalescent. The new
edges represent the blocks formed at the transition at time
$l_1$. Then, join all edges (not yet connected to two nodes) in a new
node that correspond to the set(s) of blocks merged at the second
transition and start again new edges (length 0) from these nodes
(representing the newly formed blocks at the second
transition). Repeat this until all blocks are merged, which
corresponds to the root of the tree.

Kingman's $n$-coalescent has been introduced by Kingman
\cite{FF-Kingman1982a,FF-Kingman1982b} as approximating the
genealogical tree of a sample of $n$ individuals of a haploid
Wright--Fisher model with (large) population size $N$.  Here, the
discrete genealogy is defined by recording the partitions of
$\{1,\ldots n\}$ corresponding to which sets of individuals in
$\{1,\ldots n\}$ have the same ancestor $r$ generations backwards from
the time of sampling for all $r\in\NN$. This partition-valued process
$(\cR^{(N)}_r)_{r\in\NN}$ then converges for $N\to\infty$ when time is
rescaled, i.e.
\begin{equation}\label{FF-eq:conv2coal}
\big(\cR^{(N)}_{\lceil c^{-1}_Nt\rceil}\big)^{}_{t\geq 0}\stackrel{d}{\to} \Pi^{(n)},
\end{equation}
where $\Pi^{(n)}$ is Kingman's $n$-coalescent $\Pi^{(n),\delta_0}$ and
$c_N$ is the probability that two individuals drawn at random from the
same generation have the same parent in the reproduction model, so
$c_N=N^{-1}$ in the Wright--Fisher model.\index{Wright--Fisher!model}
For any $\Xi$-$n$-coalescent
$\Pi^{(n),\Xi}$, there exists a series of Cannings models,\index{Cannings!model}
i.e. haploid reproduction models of a population of fixed size $N$ in
any generation with exchangeable offspring numbers (i.i.d. across
generations), so that Eq.~\eqref{FF-eq:conv2coal} holds for
$\Pi^{(n)}=\Pi^{(n),\Xi}$, see \cite{FF-Moehle2001}.

While the Wright--Fisher model is a staple model in population
genetics, the standard model of a neutrally evolving fixed-size
population, the Cannings models leading to other $\Xi$-$n$-coalescents
are used more rarely. This is inherited by the coalescent process
limits. The prevalence of Kingman's $n$-coalescent also stems from its
robustness: If a series of Cannings models with $c_N\to 0$ and
non-extreme variance in offspring numbers across individuals satisfies
Eq.~\eqref{FF-eq:conv2coal} for some limit process $\Pi^{(n)}$, then
$\Pi^{(n)}$ is Kingman's $n$-coalescent, see
\cite{FF-Moehle2000}. Such models include the Moran model, where one
random individual has two, a different one no and all others one
offspring.

However, there is growing evidence that for certain populations,
$\Xi$-$n$-coalescents are fitting better as genealogy
models. Necessarily, as discussed above, these need to be populations
where ancestors are shared by many offspring. For instance, a
biological mechanism leading there is sweepstake
reproduction\index{sweepstake reproduction}, i.e. one ``lucky''
individual/genotype produces considerably more offspring on the
coalescent time scale than the rest of the (potential) parents
\cite{FF-Hedgecock2011}. An example would be type III survivorship in
marine species, where individuals reproduce with large offspring
numbers, but high early-life mortality keeps the population size
constant.  This has been modelled in \cite{FF-Schweinsberg2003} as
sampling $N$ actual offspring from the potential offspring (in reality
offspring dying in an early life stage) of $N$ parents given by
i.i.d. $(\nu^{(N)}_1,\ldots,\nu^{(N)}_N)$ with heavy tails,
i.e. $P(\nu_1\geq k)\sim Ck^{-\alpha}$.  If $\alpha\in[1,2)$, the
limit in Eq.~\ref{FF-eq:conv2coal} from this model is a
Beta-$n$-coalescent\index{coalescent!beta!$n$-}, i.e. a
$\Lambda$-$n$-coalescent with $\Lambda = Beta(2 - \alpha,\alpha )$ being a
Beta distribution. These coalescents have no dust, and come down from
infinity for $\alpha>1$ but stay infinite for $\alpha=1$. While the
Beta-$n$-coalescents model haploid sweepstake reproduction, diploid
reproduction leads to a $\Xi$-$n$-coalescent limit, the
Beta-$\Xi$-$n$-coalescent, see \cite{FF-Blath201636,FF-birkner2018coalescent}.

In \cite{FF-Eldon2006}, a different haploid model of sweepstake
reproduction in a population of fixed size $N$ was introduced: In a
Moran model,\index{Moran!model} with small probability
$N^{-\gamma}$, $\gamma> 0$, the
single individual having two offspring instead has
$U=\lceil N\Psi\rceil$ offspring (while $U-1$ instead of one parent
have no offspring) for $\Psi\in(0,1)$. If $\gamma\leq 2$, the limit in
Eq.~\eqref{FF-eq:conv2coal} is a $\Lambda$-$n$-coalescent with
$\Lambda=\frac{2}{2+\Psi^2}\delta_0+\frac{\Psi^2}{2+\Psi^2}\delta_{\Psi}$
if $\gamma=2$ and $\Lambda=\delta_{\Psi}$ otherwise. The latter class is
called Psi- or Dirac-$n$-coalescents\index{coalescent!Dirac}.

For the class of Beta-$n$-coalescents\index{coalescent!beta!$n$-},
recent studies showed evidence that for samples from Japanese sardines
and Atlantic cod, Beta-$n$-coalescents (or their diploid
$\Xi$-$n$-coalescent counterparts) are fitting models, see
\cite{FF-Steinrue,FF-niwa2016reproductive,FF-Arnason2014,FF-Blath201636}.
This needs a link of
the usually unobserved genealogy to the observed genetic data.  Most
samples only include information about leaves of the genealogy at time
$t=0$. The link is usually given by modelling the differences in the
DNA sequences of the sample at a region in the genome by tracing back
the genealogy and the mutations upon it. Mutations on a branch are
inherited by every leaf subtended by the branch and are interpreted
under the infinite-sites model,\index{infinite-sites!model}
i.e. each mutation hits another position of the sequence. The mutation
process is given by a homogeneous Poisson point
process\index{Poisson!point process}\index{process!Poisson!point}
with rate $\frac{\theta}{2}$ on the
branches of the coalescent tree, independent of this tree.

Another biological mechanism that may lead to multiple-merger
genealogies is selection. In models where reproductive success of
individuals depend on their fitness and in order to survive and
produce offspring with good survival chances, they need to stay close
enough to an ever increasing fitness ``threshold'' (which may depend on
the fitness of other individuals), the Bolthausen--Sznitman
$n$-coalescent\index{coalescent!Bolthausen--Sznitman!$n$-}, a
$\Lambda$-$n$-coalescent with $\Lambda$ being the uniform distribution
on $[0,1]$, emerges as a suitable limit genealogy model, see e.g.
\cite{FF-Neher2013,FF-Brunet2013,FF-berestycki2013,FF-Desai2013,FF-Schweinsberg2017}.  The latter two
sources consider a model where the population gets fitter by pooling
beneficial mutations of equal additive fitness gains, the genealogy
limit there is a Bolthausen--Sznitman $n$-coalescent with its external
branches elongated by a deterministic interval.  Such types of
selection are summarised under the term 'rapid selection':\index{selection!rapid} if an
individual by chance has a very large fitness advantage over the rest
of the population, its number of descendants is boosted long enough so
that a multiple merger can appear, while over time this fitness
advantage is erased (so the $n$-coalescent stays exchangeable).  The
Bolthausen--Sznitman $n$-coalescent is also a Beta-$n$-coalescent if
one chooses $\alpha=1$.  Moreover, it can be constructed from a random
recursive tree\index{tree!recursive} with $n$ nodes and its properties
can thus be linked to the Chinese restaurant
process,\index{Chinese restaurant process}
\index{process!Chinese restaurant} see
\cite{FF-goldschmidt2005random}.  For instance, the partition blocks
added to the partition block including 1, starting in $\{1\}$ and
eventually reaching $[n]$, can be seen as merging each of the $K$
cycles of a random permutation\index{permutation} of $\{2,\ldots,n\}$ (a Chinese
restaurant process with $n-1$ customers) with the block containing 1
at i.i.d. $\mathrm{Exp}(1)$ times.

For the Bolthausen--Sznitman $n$-coalescent, some genomic data sets
from populations likely under strong selection have been argued to
show patterns explained by the Bolthausen--Sznitman $n$-coalescent
\cite{FF-nourmohammad2018clonal,FF-Roedelsperger2014},
although no strict model testing as for the genealogy models for
marine species listed above has been performed.

These are the two main scenarios where $\Xi$-$n$-coalescents are
reasonable genealogy models. Other $\Xi$-$n$-coalescents also appear
as genealogies in several additional contexts, see e.g. the reviews
\cite{FF-Tellier2014} and \cite{FF-irwin2016importance}.

This all shows that $\Lambda$- and $\Xi$-$n$-coalescents are a
mathematically diverse class of Markov processes with a range of
(potential) applications, but with sparse evidence so far for which
range of species/populations they should be used as the standard
sample genealogy model. Since Kingman's $n$-coalescent has been the
standard model for genealogies, with extensions e.g. for including
population size changes and population structure, many population
genetic predictions about populations are based on properties of
Kingman's $n$-coalescent. Thus, if the genealogy is given by a
different $\Xi$-$n$-coalescent, these properties may change
profoundly, with consequences for interpreting and handling the
genetic diversity of such populations.

\section{Modelling multiple mergers for variable population size
}\label{FF-sec:MMCpopsized}

The pre-limiting Cannings models whose genealogies converge to
$\Lambda$- and $\Xi$-$n$-co\-a\-les\-cents are models of a sample taken from
\emph{one generation} in \emph{one population} of \emph{fixed size
  $N$} (or $2N$ in the diploid case) across time. Moreover, the
mutational model produces \emph{linked mutations}, i.e. the genealogy
stays the same for any part of the genomic region modelled. Any of
these assumptions can (and often will) be violated for real
populations. However, a series of extensions both for haploid and
diploid models are available, some also developed within this Priority
Programme, see the chapters by Sturm~\cite{FF-AS20}, Birkner and
Blath~\cite{FF-MBJB20}, and Kurt and Blath~\cite{FF-JBNK20} in this volume. 
For $\Lambda$-$n$-coalescents, recombination has been added in
\cite{FF-birkner2013ancestral}. An approach for accounting for serial
sampling, i.e. sampling in different generations has been proposed in
\cite{FF-Hoscheit356097}. A general model for genealogies in diploid
populations and any combination of standard positive selection,
population structure through demes connected by migration, population
size changes and recombination has been introduced in
\cite{FF-koskela2018robust}.  Their model extends the model for
genealogies of diploid individuals when offspring distributions are
skewed from \cite{FF-birkner2018coalescent}, which satisfies
Eq.~\eqref{FF-eq:conv2coal} with $\Pi^{(n)}$ being the
Beta-$\Xi$-$n$-coalescent. Our focus here lies on modelling
population size changes for $\Lambda$-$n$-coalescents as limits of
haploid Cannings models. For Kingman's $n$-coalescent
$(\Pi^{(n),\delta_0}_t)_{t\geq 0}$, constructed as the limit of
genealogies in the Wright--Fisher model, population size changes of
order $N$ in the Wright--Fisher model lead to a time-changed
coalescent limit, see \cite{FF-griffiths1994sampling, FF-Kaj2003}.
In more detail, assume the size $N_r$ of the
population $r$ generations before sampling in a Wright--Fisher model
($r\in\NN_0$) is deterministic and can be characterised, for
$N=N_0\to\infty$, by an existing positive limit
\begin{equation}\label{FF-eq:coaltimeNchange}
\nu(t)=\lim_{N\to\infty}\frac{N^{}_{\lceil tc_N^{-1}\rceil}}{N}>0
\end{equation}
for all $t\geq 0$. The limit of the discrete coalescents of $[n]$,
rescaled by $c_N=N^{-1}$ from the fixed population size model, is then
\begin{equation}\label{FF-eq:timechanged_coal}
\big(\cR^{(N)}_{[ c^{-1}_Nt]}\big)^{}_{t\geq 0}\xrightarrow{\ d\ }%stackrel{d}{\to}
  \big(\Pi^{}_{g(t)}\big)^{}_{t\geq 0} 
\end{equation}
(in the Skorohod-sense) for $g(t)=\int_0^t \nu(s)^{-1} \dd s$ and
$(\Pi_t)_{t\geq 0}=(\Pi^{(n),\delta_0}_t)_{t\geq 0}$.  For instance,
consider the specific case
\begin{equation}\label{FF-eq:expgrowth}
  N^{}_r=\lfloor N(1-c_N\rho)^r\rfloor \mbox{ for } r\in\NN^{}_0,
  \rho\geq 0 \Rightarrow \nu(t)=\ee^{-\rho t},
\end{equation}
i.e. on the timescale of the coalescent limit we have exponential
growth of the population size with rate $\rho$. For this scenario, we
call the limit process Kingman's $n$-coalescent with exponential
growth. Additionally, for the Wright--Fisher model with non-constant
population sizes, the genealogical relationship is clearly
established: it is still defined as `offspring chooses parent at
random from parent generation'.

Similarly time-changed $\Lambda$- and $\Xi$-$n$-coalescents have been
proposed as reasonable genealogy models for populations with skewed
offspring distribution and fluctuating population sizes, e.g. in
\cite{FF-Spence2016}.  However, an explicit construction as a limit of
Cannings models with fluctuating population sizes has only been shown
for the special case of Dirac $n$-coalescents with exponential growth on
the coalescent time scale, i.e. population size changes given by
\eqref{FF-eq:expgrowth} in the underlying Cannings models.  To present
an explicit construction in general (for more multiple-merger
$n$-coalescents and more general population size changes) is not
(always) as straight forward as for Wright--Fisher models. The
convergence itself is essentially covered by the machinery from
\cite{FF-mohle2002coalescent}, but some care is needed for defining
the genealogical relationship.  The idea from
\cite{FF-Matuszewski2017} is to use a specific class of Cannings
models, the modified Moran models, where the genealogical relationship
can be easily established and the convergence can be shown by
verifying the conditions from \cite{FF-mohle2002coalescent}.
\index{Moran!model!modified} The modified Moran model is defined as
follows from parent to offspring generation. In the parent generation,
one individual is picked at random that has $U\geq 2$ offspring.  The
standard Moran model fixes $U=2$, where the modified Moran model sets
this as a random variable. $U-1$ randomly chosen other individuals
from the parent generation have no offspring, while the remaining
parent individuals have exactly one offspring. For fixed population
size, in order that there can be a continuous time coalescent limit
for the discrete genealogies, one needs $c_N\to 0$ for
$N\to\infty$. For modified Moran models, this is equivalent to
$N^{-1}U_N\xrightarrow{\ d \ }0$ for $N\to\infty$,
see~\cite{FF-huillet2013extended}.

An example of such a modified Moran model was shown in the first
section, the sweepstake reproduction model from
\cite{FF-Eldon2006}. Based on this model, \cite{FF-Matuszewski2017}
added exponential growth as in Eq. \eqref{FF-eq:expgrowth}, by having
one individual $r$ generations back from sampling have
$U_{N,r}=\max\{N_{r-1}-N_r,2\one_{R_r> N_r^{-\gamma}}+\lceil
N_r\Psi\rceil \one_{R_r\leq N_r^{-\gamma}}\}$ individuals as offspring
for $0<\gamma<2$, $\Psi\in(0,1)$, $(R_r)_{r\in\NN}$ i.i.d. uniform on
$[0,1]$. In the same parent generation, a random set of $U_{N,r}-1$
other individuals have no offspring, while all others have exactly one
offspring each. Then, as proven in \cite{FF-Matuszewski2017},
\cite[Thm. 2.2]{FF-mohle2002coalescent} ensures that
\eqref{FF-eq:timechanged_coal} holds for a time-changed Dirac
coalescent $(\Pi^{(n),\delta_p}_{g(t)})_{t\geq 0}$ as limit, but for
$g(t)=\int_0^t \exp(-\rho s)^{-\gamma} \dd s$.  This dependence of the
timescale on the specific way the Cannings model is defined is
noteworthy, since for fixed population size across generations, the
coalescent limit of the discrete genealogies is not changed by the
specific choice of $\gamma$. This may also prove problematic for
parameter inference, since inferring the growth parameter $\rho$
requires knowledge of $\gamma$.  This nullifies one strength of many
coalescent approaches, that the specificities of the pre-limit models
do not change the coalescent limit. However, at least calibrating the
per-generation mutation rate $\mu_N$ in the discrete Cannings model, 
which leads to a mutation rate
$\theta/2=\lim_{n\to\infty}c_N^{-1}\mu_N$ in the coalescent limit, does
not depend on the fluctuation of the population sizes, but just on the
Cannings models.

In \cite{FF-freund2019cannings}, I could show that this approach can
be generalised to allow much more general modified Moran models and
other Cannings models with fluctuating population sizes whose
genealogies converge, after rescaling by $c_N=c_{N_0}$, to a
timescaled $\Lambda$-$n$-coalescent, as in
Eq.~\eqref{FF-eq:timechanged_coal}\index{coalescent!Lambda@$\Lambda$-$n$-!with time change}. Essentially, this boils down to verifying that
\cite[Thm. 2.2]{FF-mohle2002coalescent} can be applied.  For any
$\Lambda$, one can use the modified Moran models constructed in
\cite[Prop.~3.4]{FF-huillet2013extended}, whose genealogies for
constant population size across generations converge to the
$\Lambda$-$n$-coalescent with the usual rescaling of time.
 These are defined
either by distributing $U_N$ as the size of the block being produced
at the first merger of a $\Lambda$-$N$-coalescent (variant $A$) or by
using $U_N 1_{B}+2(1-1_{B})$ instead, where $1_{B}$ is independent of
$U_N$ with $P(B)=N^{-\gamma}$ for $\gamma\in(1,2)$ (variant $B$). For
changing population sizes, it now needs to be defined how the
genealogical relation between parents and offspring is within this
model.  Let $U_{N_r}$ be the offspring in generation $r$ for the
modified Moran model with constant size $N_r$. If there is a reduction
in size from $N_r$ to $N_{r-1}$ from one generation to the next, one
can just sample down from the $N_r$ offspring that the fixed size model
produces, resulting again in a modified Moran model.  For increases in
population size, the model stays a modified Moran model if additional
individuals are added as further offspring of the parent having
$U_{N_r}$ offspring or as single offspring of parents not reproducing
in the fixed size model (since $U_N\geq 2$, at least one individual
can be added with the second method).  Then, for any $\Lambda$ and for
any time change function $\nu$ as in Eq.~\eqref{FF-eq:coaltimeNchange}
there exist such modified Moran models with population sizes
$(N_r)_{r\geq 0}$ satisfying Eq.~\eqref{FF-eq:coaltimeNchange} so that
the discrete genealogies, properly scaled with a time-inhomogeneous
function (related to $c_N$, but not just scaling by $c_N$), converge
to a $\Lambda$-$n$-coalescent limit. Under mild additional assumptions
this is equivalent to the convergence as described in
Eq.~\eqref{FF-eq:timechanged_coal} with a different time change $g$;
for a certain class of $\nu$ and $\Lambda$ one needs no additional
assumptions at all. See the following two sample results from
\cite{FF-freund2019cannings} (Theorem~1 and Corollary~1
therein). The first one includes the case of modified Moran models
converging to a time-changed Dirac coalescent from
\cite{FF-Matuszewski2017}. In both propositions, additional
individuals can be added in any way so that the resulting model is
still a modified Moran model.
\begin{proposition}
  Let $U_N$ be distributed as the first jump of a
  $\Lambda$-$N$-coalescent. If
  $(N-1)^{-1}E(U_N(U_N-1))\nrightarrow 0 \mbox{ for } N\to\infty$,
  define the modified Moran model via variant $B$. Then, for any
  positive function $\nu$ there exist population sizes satisfying
  Eq. \eqref{FF-eq:coaltimeNchange} so that the discrete genealogies
  of the modified Moran model with variable population sizes converge
  as in Eq. \eqref{FF-eq:timechanged_coal} with
  $g(t)=\int^t_0 (\nu(s))^{-\gamma} \dd s$ and $(\Pi_t)_{t\geq 0}$ is a
  $\Lambda$-$n$-coalescent.
\end{proposition}

\begin{proposition}\index{coalescent!beta!$n$-!with time change}
  Let $\Lambda$ be a $Beta(a,b)$-distribution with $a\in(0,1)$ and
  $b>0$. Let $\nu:\mathbb{R}_{\geq 0}\to \mathbb{R}_{>0}$.  Then,
  there exist population sizes satisfying
  Eq.~\eqref{FF-eq:coaltimeNchange} for $\nu$ so that the genealogies
  $(\cR^{(N)}_r)_{r\in\mathbb{N}_0}$ of the modified Moran model with
  variable population sizes (variant $A$) fulfill
  Eq.~\eqref{FF-eq:timechanged_coal} with
  $g(t)=\int^t_0 (\nu(s))^{a-2} \dd s$ and $(\Pi_t)_{t\geq 0}$ is a
  $\mathrm{Beta}(a,b)$-$n$-coalescent.
\end{proposition}

For the standard Moran model, the second proposition holds for $a=0$
and the limit coalescent is a time-changed Kingman-$n$-coalescent, see
\cite[Prop. 1]{FF-freund2019cannings}. Thus, when comparing to
Eq.~\eqref{FF-eq:timechanged_coal} for the Wright--Fisher model, the
same population size changes $\nu$ on the coalescent time scale lead
to different time changes of the limit Kingman 
$n$-coalescent. Similarly, for a given $\Lambda$ and $\nu$, there can
be other Cannings models whose genealogies, for constant population
size, have the same coalescent limit, but where having the same
population size profile $\nu(t)$ on the coalescent time scale still
leads to different time scalings $g$.  For instance, consider
Schweinsberg's Cannings model from \cite{FF-Schweinsberg2003}, whose
genealogies converge to Beta-$n$-coalescents, as described in the
first section.  The potential offspring produced from a parent
generation in the fixed-size setting is enough, at least with very
high probability $\to 1$ for $N\to\infty$, to also sample an increased
number of offspring from these, for any population size changes
allowed by \eqref{FF-eq:coaltimeNchange}.  For decreasing population
size, just sampling less offspring of course also works. Again,
\cite[Thm. 2.2]{FF-mohle2002coalescent} can be applied. If, for
constant population size, the discrete genealogies of $[n]$ converge
to the $Beta (2-\alpha,\alpha)$-$n$-coalescent for $1\leq \alpha < 2$,
when adding population size changes described by $\nu$, the time
change in Eq.~\eqref{FF-eq:timechanged_coal} is given by
$\int_0^t (\nu(s))^{1-\alpha} \dd s$, see
\cite[Thm. 2.8]{FF-freund2019cannings}. Again, this is a different
timescale than in the case of the modified Moran models above. Even
more disturbingly, for $\alpha=1$ the distribution of the limit
genealogy is invariant under any population size change allowed by
Eq.~\eqref{FF-eq:coaltimeNchange}.
   
\section{How much genetic information is contained in a subsample?}

For managing genetic resources in crops, gene banks hold many
accessions of a crop, e.g. as seeds (here accession can be understood
as individual plant sampled at random, see~\cite{FF-FAO} 
%\url{http://www.fao.org/wiews/glossary/en/}
for a thorough
definition). However, resources of gene banks are limited and thus not
all individual plants can be stored. Individuals grown from such seeds
can be used as crossing partners to introduce genetic variation not
yet present into a breeding program. In terms of genealogies, new
variation is available if the genalogy of the material already in the
breeding program together with the gene bank material has additional
branches with mutations as compared to the genalogy of the material used
in the breeding program. How much do different genealogy models affect
this amount of added genetic variation?

Mathematically, this corresponds to assessing the overlap of the
nested genealogies of a sample of size $n$ and an arbitrary subsample
of size $m$ (and the mutations on the non-overlapping
branches). Assume that the genealogy of the $n$ sampled individuals is
given by a $\Xi$-$n$-coalescent. From the Poisson construction, it
follows that the subsample is a $\Xi$-$m$-coalescent with the same
measure $\Xi$, a property called natural coupling. Due to
exchangeability, for questions concerning the distribution of sharing
aspects of the genealogy it does not matter which individuals are in
the subsample and in the sample, thus we always assume the subsample
of size $m$ is $[m]$\index{coalescent!Xi@$\Xi$-$n$-!subsampling}.

For Kingman's $n$-coalescent, questions about how much of the
genealogy is covered have been discussed in
\cite{FF-saunders1984genealogy}. For instance, the subsample's
genealogy covers the root of the sample's genealogy with probability
$p^{(\delta_0)}_{n,m}=\frac{m-1}{m+1}\frac{n+1}{n-1}$ for
$m\leq n\in\mathbb{N}$. If the root is shared, any mutations on the
two branches starting in the root are also present in the
sample. Together with Eldon~\cite{FF-Eldon2018}, I showed
that this probability can also be computed recursively for any
$\Lambda$-$n$-coalescent as
\begin{equation}\label{FF-eq:pnmrec}
  p_{n,m}^{(\Lambda)} =  \sum_{k=2}^n \lambda(n,k)\sum_{\ell = 0}^{k\wedge m}
  \frac{\binom{n-m}{k-\ell}
    \binom{m}{\ell}}{\binom{n}{k}}p_{n-k+1,m^\prime}^{(\Lambda)},
\end{equation}
where
\[
  \lambda(n,k)=\frac{\binom{n}{k}\lambda^{}_{n,k}}{\sum_k \lambda^{}_{n,k}},
  \quad 
  \lambda^{}_{n,k}=\int_0^1 x^{k-2} (1-x)^{n-k} \Lambda(\dd x) 
  \]
is the probability of transition by merging any $k$ of $n$ blocks present and
$m^\prime = (m-\ell + 1)\one_{\{\ell > 1\}} + m\one_{\{\ell \le
  1\}}$. As boundary conditions, we have $p^{(\Lambda)}_{k,k} = 1$ for
$k\in\mathbb{N}$ and $p^{(\Lambda)}_{k,1}=0$ for $k\geq 2$. As many
recursions for $\Lambda$-$n$-coalescents, this can be proven by
conditioning on the first jump $T_1$ of the $\Lambda$-$n$-coalescent:
The genealogy cut at $T_1$, keeping all branches connected to the
root, is again a $\Lambda$-$k$-coalescent tree, where $k$ is the
number of blocks/branches present at $T_1$. Eq.~\eqref{FF-eq:pnmrec}
then only sums over all possible mergers of $k$ blocks, where $\ell$
blocks of the subsample and $k-\ell$ of the sample without the
subsample are merged. If $\ell=m$, the root of the subtree of $[m]$ is
reached, and thus the root is shared if and only if no blocks in
$[n]\setminus [m]$ are unmerged, which is encoded by the boundary
condition. Using the recursion, we can compare $p^{(\Lambda)}_{m,n}$
between $\Lambda$-$n$-coalescents, see Figure~\ref{FF-fig:pnm}.
\begin{figure}[t]
\begin{center}
\includegraphics[width=0.7\textwidth]{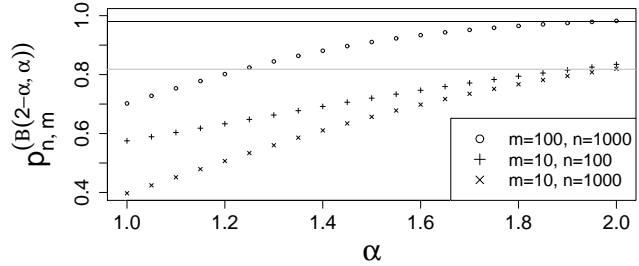}
\end{center}
\caption{\label{FF-fig:pnm} The probability $p^{(\Lambda)}_{n,m}$ of
  sharing the root between sample and subsample for
  $Beta (2-\alpha,\alpha)$-$n$-coalescents. Solid lines show the
  asymptotic probabilities $\lim_{n\to\infty} p^{(\delta_0)}_{n,m}$
  for Kingman's $n$-coalescent for $m=10$ (in grey) and $m=100$ (in
  black).}
\end{figure}
% Similar recursions hold for other properties of sharing genealogical
% properties between sample and subsample, e.g. the probability of
% sharing the oldest mutant allele, see \cite[Sect. ]{FF-Eldon2018}.
Fig.~\ref{FF-fig:pnm} may raise the question whether Kingman's
$n$-coalescent is the $\Xi$-$n$-coalescent that has maximum
$p_{n,m}^{(\Xi)}$. This is not true, for instance
$\Xi$-$n$-coalescents that are essentially star-shaped, i.e. merge all
individuals at the first merger with high probability, have a higher
$p_{n,m}^{(\Xi)}$ if the probability for being star-shaped is high
enough.

Due to the exchangeability of the $\Xi$-$n$-coalescent,
$p^{(\Xi)}_{n,m}$ can alternatively be expressed in terms of the block
sizes of the exchangeable partition of the $\Xi$-$n$-coalescent
shortly before absorption.

\begin{proposition}\cite[Props. 2 and 5]{FF-Eldon2018} 
 For any $\Xi$-$n$-coalescent,
\begin{equation}\label{FF-eq:prob_nsameMRCA}
p^{(\Xi)}_{n,m}= 1 -  E\Big(\sum_{i\in\NN}\prod_{k=0}^{m-1}
\frac{B^{(n)}_{[i]}-k}{n-k}\Big) >0,
\end{equation}
where $B^{(n)}_{[1]},B^{(n)}_{[2]},\ldots$ are the sizes of the blocks
of $\Pi^{(n)}$ merged at absorption, ordered by size from largest to
smallest (filled up by empty blocks). The limit
$p_m^{(\Xi)}:=\lim_{n\to\infty}p^{(\Xi)}_{(n,m)}$ exists and is greater 
than $0$ if and only if the $\Xi$-coalescent does not stay infinite. If
the $\Xi$-coalescent comes down from infinity, we have
\begin{equation}\label{FF-eq:rep_p_lastcollfreq}
  p^{(\Xi)}_{n,m}\to  1 - E\Big(\sum_{i\in\NN}P_{[i]}^m\Big) =
  1- E(X^{m-1})=1-\frac{ E(Y^m) }{E(Y)} >0
\end{equation}
for fixed $m$ and $n\to\infty$, where
$P_{[i]}:=\lim_{n\to\infty} B^{(n)}_{[i]} / n$ is the asymptotic
frequency of the $i$-th largest block merged at absorption, $X$ is the
asymptotic frequency of a size-biased pick and $Y$ is the asymptotic
frequency of a block picked uniformly at random from the blocks
merging at absorption.
\end{proposition}

\begin{beweis} This is a condensed version of the proof, focusing on
  the ideas and omitting technicalities.
  
  To see Eq.~\eqref{FF-eq:prob_nsameMRCA}, first observe that
  $1-p^{(\Lambda)}_{(n,m)}$, the probability that the root is not
  shared, means that all individuals of $[m]$ have merged before the
  last merger. Thus, we need to have $[m]\subseteq B^{(n)}_{[i]}$ for
  some block of the partition merged at absorption. Given
  $B^{(n)}_{[1]},\ldots,B^{(n)}_{[n]}$, the probability of
  $[m]\subseteq(B^{(n)}_{[i]}$ 
  is $\sum_{i=1}^{n}\prod_{\ell=0}^{m-1}$ $\big(B^{(n)}_{[i]}-
  \ell\big)/(n- \ell)$ for an $i\in[n]$ 
  due to exchangeability.
    
  This allows us to consider
  $p_m^{(\Lambda)}:=\lim_{n\to\infty}p^{(\Lambda)}_{(n,m)}$ for fixed
  $m$. From the Poisson construction, one sees that adding more
  individuals can only add branches and nodes to the tree. Thus
  $p^{(\Lambda)}_{(n,m)}$ is monotonic in $n$ for $m$ fixed. When is
  the thus existing limit $p_m^{(\Xi)}=0$? If the $\Xi$-coalescent
  stays infinite, the waiting time for absorption, the height of the
  genealogy, diverges almost surely for $n\to\infty$. Thus, the root
  cannot be shared almost surely, so $p_m^{(\Xi)}=0$. Consider now
  the case that the coalescent does not stay infinite.  As discussed above,
  this also means that the heights of the $\Xi$-$n$-coalescents equal
  the height of the $\Xi$-coalescent for $n$ large enough.
  Exchangeability (and elementary, yet technical
  arguments) ensures that the asymptotic frequencies of the blocks
  participating in the last merger exist. There are at least two
  blocks, no block has frequency 1, and again exchangeability ensures
  that with positive probability, $[m]$ is not a subset of a single
  block. This shows $p^{(\Xi)}_m>0$ if the coalescent does not stay
  infinite.  The convergence in Eq.~\eqref{FF-eq:rep_p_lastcollfreq}
  follows directly due to the link between block sizes and asymptotic
  frequencies, while the characterisations using the size-biased and
  uniform picks are just standard characterisations of asymptotic
  frequencies in exchangeable partitions on $\NN$.
\end{beweis}

If one knows the moments of the asymptotic frequencies of the blocks
of $\Pi^{(\Xi)}$ participating in the last merger,
Eq.~\eqref{FF-eq:rep_p_lastcollfreq} would give an explicit formula
for $p_m^{(\Xi)}$. For Beta-coalescents that come down from
infinity\index{coalescent!beta!$n$-}, i.e. $1<\alpha<2$, there is an
explicit characterisation of asymptotic frequencies, conditioned on
the event that the Beta-coalescent is in a state with $k$ blocks,
$k\in\mathbb{N}$. In this case, the asymptotic frequencies can be
expressed in terms of Slack's distribution,\index{Slack's distribution}\index{distribution!Slack's} see
\cite[Thm. 1.2]{FF-berestycki2008small}. Since the distribution of the
number of blocks $K$ participating in the last merger of this class of
coalescents is also known, see \cite{FF-henard2015fixation}, we can
represent $p^{(\Xi)}_m$ as
\begin{equation}\label{FF-eq:pmbeta}
p^{(Beta (2-\alpha,\alpha))}_m=1-\sum_{k\in\NN} 
 k \frac{E\left( \frac{Y_1^m}{ (Y^{}_1+\ldots +
    Y^{}_k)^{\alpha + m - 1} } \right)}{E\left((Y^{}_1+\ldots +Y^{}_k)^{1-\alpha}
   \right)}P(K=k),
\end{equation}
where $K$ has generating function
$E(u^K) = \alpha u \int_0^1 (1 - x^{1 - \alpha})^{-1} ( (1 -
ux)^{\alpha-1} - 1)\dd x$ for $u \in [0,1]$ and $(Y_i)_{i\in\NN}$ are
i.i.d. and have Slack's distribution with Laplace transform
$ E\left(e^{-\lambda
    Y_1}\right)=1-(1+\lambda^{1-\alpha})^{-1/(\alpha-1)}$.

For $\alpha=1$, so for the Bolthausen--Sznitman ($n$-)coalescent, a
more explicit representation can be derived from the connection to the
uniform random permutation, aka the Chinese restaurant
process.\index{Chinese restaurant process}\index{coalescent!Bolthausen--Sznitman!$n$-}
\index{process!Chinese restaurant}

\begin{proposition}{\cite[Prop. 4]{FF-Eldon2018}}
  Let $B_1,\ldots,B_{n-1}$ be independent Bernoulli random variables
  with $P(B_i = 1) = i^{-1}$. For $2 \leq m< n$,
\begin{equation}\label{FF-eq:pnm_bsz}
 p_{n,m}^{(Beta (1,1))} =
E\left(\frac{B^{}_1+\ldots+B^{}_{m-1}}{B^{}_1+\ldots+B^{}_{n-1}}\right).
\end{equation}
Moreover, 
\[
\log(n) p_{n,m}^{(Beta (1,1))} \to \sum^{m-1}_{i=1} i^{-1} \text{ for } 
n\to\infty \text{ and } m \text{ fixed.}
\]
\end{proposition}

\begin{beweis}
  The last merger of the Bolthausen--Sznitman $n$-coalescent has to
  feature one block including the individual 1. As described in the
  first section, the connection to the random recursive tree allows us
  to see the individuals subsequently merged to the block including 1 as
  adding the cycles of a uniform random permutation of
  $\{2,\ldots,n\}$ in random order.  Since $1\in[m]$, the root is
  shared between subsample and sample if and only if the cycle added
  at the last merger also includes at least one $j\in[m]$. Using
  e.g. the Chinese restaurant process construction of the uniform
  random permutation, one sees that there are $\sum_{i=1}^{n-1} B_i$
  cycles in the random permutation, from which $\sum_{i=1}^{m-1} B_i$
  contain $j\in[m]$. This shows Eq.~\eqref{FF-eq:pnm_bsz} and standard
  arguments establish its limit for $n\to\infty$.
\end{beweis}

So what do these results imply for real populations? The main message
is indeed Figure~\ref{FF-fig:pnm}: Multiple merger genealogies, unless
essentially star-shaped, will have a (often much) higher chance that
enlargening collections of individuals uncovers new ancestral
variation that then can be used for breeding. Adding population size
changes of order $N$ only changes the distribution of branch
lengths but not the tree topology (see Sect.~\ref{FF-sec:MMCpopsized})
for all $\Lambda$-$n$-coalescents, so this holds true if we consider
e.g. exponentially growing populations. However, if we are concerned
about the actual shared genealogy, the picture changes somewhat.
Using simulations, we assessed the fraction of internal branch length
of the genealogy of $[n]$ that is already covered by the genealogy of
$[m]$, see~\cite[Fig. 3-6]{FF-Eldon2018}. On average, for Kingman's
$n$-coalescent a larger fraction of internal branches is covered by
the subsample's genealogy than for Beta-$n$-coalescents, while this is
reversed for Kingman's $n$-coalescent with strong exponential growth.
Additionally, the coverage is much more variable for
Beta-coalescents. Nevertheless, at least for small to medium growth
rates, the benefit of larger samples in terms of added ancestral
variation is stronger for multiple-merger genealogies. The results
highlight that inferring the correct genealogy model can be relevant
for real-life decisions concerning the population at hand.

One more technical, yet in my opinion important point should be
discussed here. While we can characterise the limit behaviour of
$p^{(\Xi)}_{n,m}$ for $n\to\infty$, this may not be of too much
relevance for the practical application. This is not so much due to
having finite sample sizes, but it is already questionable whether for
$n$ large, the coalescent approximation still resembles the genealogy
in the Cannings model, see e.g.~\cite{FF-bhaskar2014distortion} for
addressing this issue for Kingman's $n$-coalesent.  However, due to
the monotonicity of $p^{(\Xi)}_{n,m}$ for $m$ fixed and $n$
increasing, the limit results always give a lower bound for the
probability of sharing the root between sample and subsample.

The results described in this section show that adding individuals for
$\Lambda$- and $\Xi$-$n$-coalescent will usually have the effect that
the genealogical tree is enlarged by a higher factor (increased
height, but also increased internal branch length) than for Kingman's
$n$-coalescent. This may point to a conjecture for multiple-merger
coalescents about the question posed by Felsenstein in
\cite{FF-Felsenstein2005}: ``Do we need more sites, more sequences, or
more loci?'' if it comes to population genetic inference. While for
Kingman's $n$-coalescent, sample size is less important, our results
indicate that this importance is elevated for multiple-merger
$n$-coalescents.

\section{Model selection between $n$-coalescents}\index{model selection}

During the two phases of the SPP 1590, much progress has been made on
inference methods for deciding whether multiple merger genealogies
 explain the genetic diversity of a sample better than
bifurcating genealogy models based on Kingman's $n$-coalescent, both
inside and outside of the SPP. For the former, see also the chapter
 by Birkner and Blath~\cite{FF-MBJB20}. More formally, one is
interested in performing model selection between two or more sets of
$n$-coalescent models, each endowed with a range of mutation rates,
based on the values of one or multiple statistics for measuring
genetic diversity\index{coalescent!Xi@$\Xi$-$n$-!model selection}. Many
methods focus on the site frequency spectrum
$(\xi^{(n)}_1,\ldots,\xi^{(n)}_{n-1})$ as inference statistics, where
$\xi_i^{(n)}$ is the number of mutations that are inherited by exactly
$i$ individuals in the sample, i.e. that have mutation allele count
$i$, see e.g. \cite{FF-Eldon2015,FF-koskela2018multi, FF-koskela2018robust,
FF-Matuszewski2017}.  However,
genetic data contains more information than the site frequency
spectrum (SFS) and this surplus information can also be used to
perform model selection, see e.g. \cite{FF-Kato171060} or
\cite{FF-Rice461517}. A reliable method to use multiple statistics for
model selection is Approximate Bayesian Computation (ABC).\index{approximate Bayesian computation} ABC uses a
computational version of a Bayes approach, see e.g. \cite{FF-Sunn2013}
for an introduction. Consider we want to perform model selection
between two models with equal \textit{a priori} probability for each
model. Data is represented by summary statistics.  In the simplest ABC
approach (rejection scheme), statistics are simulated $N\gg 1$ times
with model parameters drawn from each model's prior distribution (so
$2N$ sets of summary statistics are produced). Then the posterior odds
ratio between models is approximated by the ratio of the numbers of
simulations from each model that are equal/very close to the observed
value of the statistics (the quality of approximation depends e.g. on
the sufficiency of the statistics used). See~\cite{FF-Lintu2017} for a
recent review of more sophisticated approaches.

In \cite{FF-Kato171060}, a variety of standard population genetic
statistics were used in an ABC approach for model selection between
Beta-$n$-coalescents and Kingman's $n$-coalescent, both with
exponential growth on the coalescent time scale. The SFS was further
summarised by using its sum, the total number of observed mutations as
well as the $(0.1,0.3,0.5,0.7,$ $0.9)$-quantiles of the mutation allele
counts (this set of quantiles will be denoted by $AF$). Additionally,
the authors used the same range of quantiles of the Hamming distances between
all pairs (set of statistics denoted by $Ham$), of the squared
correlation $r^2$ between the frequencies in the sample for each pair
of mutations and, for a reconstruction of the
genealogical tree by a standard phylogenetic methods
(neighbor-joining), of the set of all branch lengths in the
reconstructed tree (denoted by $Phy$).  For distinguishing
Beta-$n$-coalescents from Kingman's $n$-coalescent, both also
accounting for exponential growth on the coalescent time scale, they
reported considerably lower error rates than earlier methods,
e.g.~\cite{FF-Eldon2015}, while also considering a slightly different
setup. A crucial difference is that in \cite{FF-Kato171060}, the
range of mutation rates is identical for all $n$-coalescents considered.
Since the distributions of height and total branch length differ
strongly between different $n$-coalescent models, the number of
observed mutations is also different. If statistics are used that
are non-robust to the number of mutations on the genealogy, e.g. the
number of segregating sites, this approach needs large ranges for the
mutation rate to be able to reproduce these statistics as observed,
which is not efficient for performing simulation-based inference
approaches as ABC.  If the parameter ranges are not large enough,
model classes may produce the comparable numbers of mutations only
with low probability or not at all, thus likely being discarded as the
true model class. Thus, another approach (used in the other inference
approaches mentioned) is to use mutation rates $\theta$ that produce
a number of mutations in the range of the number $s$ of mutations 
observed in the data, e.g. by setting
$\theta=\frac{2s}{E_{coal}(L_n)}$ for each specific member $coal$ of a
model class's $n$-coalescents, which is the generalised form of
Watterson's estimator\index{Watterson's estimator} for $\theta$.
Here, $L_n$ denotes the sum of
the lengths of all branches of the coalescent tree.  With 
Siri-J\'egousse, I analyse which statistics
considered by~\cite{FF-Kato171060} are actually best to distinguish
between different model classes featuring Kingman's $n$-coalescent
with exponential growth and several classes of $\Xi$-$n$-coalescents,
including ones with exponential growth \cite{FF-Freund2019a}.  For
this, we use the ABC approach for model selection based on random
forests\index{random!forest} from~\cite{FF-pudlo2015reliable}
(random forests are a widely
used machine learning approach introduced in \cite{FF-Breiman2001}).
Since it is an approximate Bayesian method, we need to specify prior
distributions on both the coalescent classes and their mutation
ranges. However, the approach from \cite{FF-pudlo2015reliable} differs
drastically from the ABC approach presented above.
\index{ABC, random-forest based}%\index{random-forest based ABC}
In a nutshell, the method simulates a set of summary
statistics $S_1,S_2,\ldots$ from each model class under its prior
distribution, then takes bootstrap samples of these simulations. From
each bootstrap sample, a decision tree is built, whose nodes have the
form $S_i>t$ or $S_i<t$ for some $t\in\RR$.  For each node, the
statistic that distinguishes best between model classes (normally
measured by the Gini index, we use a slightly different measure), from
a randomly drawn subset of statistics, is chosen as decider
$S_i$. Nodes are added until the tree divides the
bootstrap sample perfectly into sets of simulations from the same model
class. The observed data is then sorted into a model class according
to the decision tree for each bootstrap sample (so for the complete
random forest), and the model selection is then the majority class
across the forest.  In other words, instead of using that the true
model should produce summary statistics closer to the observed ones
than another model, as e.g. rejection-scheme ABC, ABC based on random
forests trains a forest of decision trees based on the simulations for
the model priors and lets them decide, which model matches the
observed statistics.

We chose this method mainly since it does not increase its model
selection error when including many and potentially uninformative
statistics and it comes with a built-in measure for the ability of
each statistic to distinguish hypotheses, the variable importance. The
variable importance of $S_i$ measures the average decline in
misclassification over all nodes of the random forest where $S_i$ was
picked as decider.  To assess the misclassification errors, the ABC
method uses the out-of-bag error. For this, one takes the fraction of
trees in the random forest for each simulation \textit{sim} that were
built without \textit{sim} and that sorted \textit{sim} into a wrong
model class (and then averages over all simulations).

As summary statistics, we use the statistics as in
\cite{FF-Kato171060}, but we add several more statistics, for instance
nucleotide diversity $\pi$, the mean of the Hamming distances between
all pairs of sampled individuals. We also add a new
statistic\index{coalescent!Xi@$\Xi$-$n$-!clade size}.
% For each individual $i\in\{1,\ldots,n\}$, consider the set
% $\cM_n(i)$ of all mutations that are shared by at least one other
% individual. The size $O_n(i)$ of the minimal observable clade is the
% smallest mutation allele count in this set of mutations, $O_n(i)=n$
% if $\cM_n(i)=\emptyset$. For any mutation with count $>1$,
% individuals sharing this mutation form a clade, i.e. they are the
% complete set of leaves of a subtree in the genealogical tree, since
% they inherited this mutation from a common ancestor. Additionally,
% clades including a specific individual are nested, i.e. one of their
% corresponding subtrees is a subtree (maybe the same tree) of the
% other one. Thus, the clade spanned by any minimum allele count
% mutation from $\cM_n(i)$ is identical and the smallest clade defined
% by a mutation from $\cM_n(i)$. Moreover, one can define $O_n(i)$
% equivalently as follows:
For each $i\in[n]$, $O_n(i)$ is the number of individuals including
$i$ that share all non-private mutations of $i$ (a mutation is private
if it is only found in one sampled individual). This correponds to the
smallest allele count $>1$ of all mutations that affect individual
$i$. See Figure~\ref{FF-fig:on} for an example. We consider the
$(0.1,0.3,0.5,0.7,0.9)$-quantiles of $(O_n(i))_{i\in[n]}$ as well as
the mean, standard deviation and harmonic mean as statistics, this set
of statistics is denoted by $O$.
\begin{figure}[t]
\begin{center}
      \includegraphics[width=0.5\textwidth]{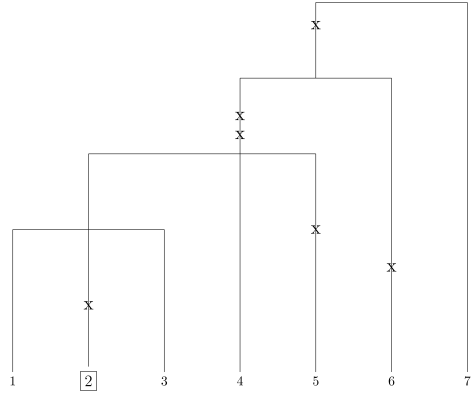}
\end{center}
\caption{\label{FF-fig:ocmc} Genealogical tree, mutations marked by
  \texttt{x}. Individual 2 is affected by 4 mutations with allele
  counts 1,5,5,6. All non-private mutations are shared by individuals
  $[5]$. The observable minimal clade size of 2 is $O_n(2)=5$.}
\label{FF-fig:on}
\end{figure}

Our first result is that, not very surprisingly, including more
statistics decreases the out-of-bag error meaningfully
across a range of different sets of $n$-coalescents.  See
Table~\ref{FF-tab:m12ABCres} for a comparison between model classes of
Beta-$100$-coalescents (class Beta) for $\alpha\in [1,2)$ and
Kingman's $n$-coalescent with exponential growth (class Growth) and
growth rate in $\rho\in[0,1000]$. We chose a uniform prior on
$\alpha$, but (essentially) a $\log$ uniform prior on the growth rate,
with an atom at $0$, to put a higher prior weight on low, more
realistic growth rates (however, some real data sets of fungal or
bacterial pathogens do fit relatively well to models with growth rates
of 500 or more).  The mutation rate $\theta$ was chosen for each model
in any model class so that it produces, on average, a number of
mutations $s\in\{15,20,30,40,60,75\}$, i.e. we draw $s$ from this set
according to a (uniform) prior distribution and then set $\theta$ to
the generalised Watterson estimate.  We performed 175,000 simulations
per model class.  From these, we conducted the described ABC
analysis. We perform the exact same analysis a second time, including
new simulations with parameters from a new (and independent) draw from
the prior distributions.

\begin{table}\begin{tabular}{llll}
Set of statistics & Misclassification for Beta & Misclassification for Growth\\ \hline
$AF$, $S$,$\pi$ & 0.333 (0.330) & 0.246 (0.246)\\ 
$AF$, $S$,$\pi$,$O$ &  0.246 (0.242) &  0.205 (0.209) \\
$AF$, $S$,$\pi$,$Ham$ & 0.277 (0.274) & 0.222 (0.227)\\ 
$AF$, $S$,$\pi$,$Phy$ & 0.317 (0.320) & 0.247 (0.244)\\ 
$AF$, $S$,$\pi$,$r^2$ & 0.298 (0.290) & 0.228 (0.235)\\ 
$AF$, $S$,$\pi$,$O$,$Ham$ &  0.241 (0.240) &  0.203 (0.205) \\
FULL - $O$ & 0.269 (0.271) & 0.217 (0.215) \\ 
FULL - $(AF,S,\pi)$ & 0.242 (0.240) & 0.208 (0.210) \\ 
FULL &  0.243 (0.240) &  0.200 (0.204) \\[2mm]
\end{tabular}
\caption{Misclassification errors (out-of-bag errors) of the random
  forest ABC model comparison (5,000 trees) with different sets of
  summary statistics. In parentheses: results for rerun of
  analysis. FULL: $AF$,$S$,$\pi$,$O$,$Ham$,$Phy$,$r^2$}
\label{FF-tab:m12ABCres}
\end{table}

A more interesting take-away from our analysis is that adding
statistics based on the minimal observable clade sizes $O$ to
the set based on the allele counts $AF$, $S$, $\pi$ leads to
considerably decreased error rates, which are only marginally reduced by
further additions of statistics. The results are essentially unchanged if the
full site frequency spectrum is used instead of $AF$. For the model
selection using all considered statistics, the variable importance was
highest for the harmonic mean of the minimal observable clade sizes.
A potential explanation why the harmonic mean of $(O_n(i))_{i\in[n]}$
is a meaningful statistic for such model selections can be found
in~\cite[p. 31-32]{FF-Freund2019a} and has to do with the connection
of $O_n(i)$ with the minimal clade size of $i$.  These findings remain 
true for other model selection problems, e.g. if one changes the
multiple merger genealogy model to a Dirac or
Beta-$\Xi$-$n$-coalescent.  The results presented here indicate that
focusing on the site frequency spectrum for model selection between
multiple merger and binary $n$-coalescents can be a suboptimal choice
and that combining the information with further statistics as also
done in~\cite{FF-Rice461517} is advisable.  A similar effect was
observed for model selection between different scenarios of population
size changes for bifurcating genealogies, see~\cite{FF-Jay2019}.  Due
to the recent advance of full-data methods for such related inference
problems, notably \cite{FF-Palacios19}, where past population sizes
are inferred based on a Kingman's $n$-coalescent with fluctuating
population sizes (essentially inferring the function $\nu$
from~\eqref{FF-eq:coaltimeNchange}) by an efficient representation of
the full data,
%(which also features tracing back clades within the tree)
a revisit of full-likelihood methods may also be a viable alternative.

% \\ The beneficial effect of adding statistics based on the minimal observable clade sizes persists in scenarios where model classes  differ in the underlying $\Xi$-$n$-coalescent (we used Kingman's, Dirac, Beta- and Beta-$\Xi$-$n$-coalescents, the first two also combined with exponential growth), while this effect vanishes if we compare the same $\Xi$-$n$-coalescent for both classes, only differing in whether growth is added or not.\\

Where is this surplus of information for model selection really
critical?  For species with large enough genomes featuring many
chromosomes and/or linkage blocks, the two
studies~\cite{FF-koskela2018multi} and \cite{FF-koskela2018robust}
suggest that the information within the SFS is already enough to
reliably distinguish between model classes over a range of different
genealogy models (see the chapter of Birkner and Blath~\cite{FF-MBJB20}
in this volume). However, for e.g.\ bacteria with low recombination rates,
where the entire genome can be essentially seen as one linkage block,
this will not work as well. Here, pooling of more information than the
SFS/AF clearly helps with distinguishing between model classes.  An
example of such bacteria is \textit{Mycobacterium tuberculosis}, the
bacterial agent of human tuberculosis, which is haploid and
propagating clonally. Genealogies from \textit{M. tuberculosis}
outbreaks are usually modelled as Kingman's $n$-coalescent with
exponential growth.  However, the bacteria have to evolve quickly due
to a high selection pressure, which could lead to a
Bolthausen--Sznitman $n$-coalescent as the suitable genealogy model.
Moreover, real data shows some patients that are super-spreaders,
i.e. infecting very many other patients, which could be modelled as a
multiple-merger genealogy. With Menardo and Gagneux, I investigated
in~\cite{FF-Menardo2019} whether indeed
classes of $\Lambda$-$n$-coalescents fit better to the data than
Kingman's $n$-coalescent with exponential growth. We considered 11~publicly
available data sets and performed model selection using the
ABC approach described above. Our results show that 8 of 11 data sets
actually fit better to multiple merger genealogies (10 of 11 if one
allows for $\Lambda$-$n$-coalescents with exponential growth) and that
these produce patterns of genetic diversity compatible with the
observed data.

\section{Partition blocks and minimal observable clades}

Seeing that the minimal observable clade sizes\index{clade size} are a
reasonable addition to the arsenal of population genetic statistics,
Siri-J\'egousse and I also investigated their mathematical
properties \cite{FF-Freund2019b}\index{coalescent!Xi@$\Xi$-$n$-!clade size}.
In the following, the key findings are presented.  Looking
back at Figure~\ref{FF-fig:ocmc}, the minimal observable clade of
individual $i$ can be represented as the block including $i$ of the
$\Xi$-$n$-coalescent at the time of the first mutation affecting $i$
that is not on the branch connecting leaf $i$ to the rest of the
tree, i.e. on the external branch of $i$.  Due to exchangeability, the
distribution of the size of the minimal observable clade is identical
for all $i\in[n]$, so in the following we fix $i=1$ and omit the $i$
for ease of notation.  Let $O_n$ be the size of the minimal observable
clade of 1 and $E_n$ the length of the external branch of 1. Since
mutations are modelled by a homogeneous Poisson point process,
independent of the $n$-coalescent, on the branches of the coalescent
tree with rate $\frac{\theta}{2}$, the waiting time for the first
non-private mutation on the path of leaf 1 to the root of the tree is
$E_n+T_n$, where $T_n$ is a exponentially distributed with parameter
$\frac{\theta}{2}$ independent of $E_n$ and of the coalescent.  If
$E_n+T_n$ exceeds the height of the $n$-coalescent tree, there is no
non-private mutation of $i$ and thus $O_n=n$.  Recall that
$B^{(n)}_1(t)$ denotes the size of the block including 1 in the
partition $\Pi^{(n)}_t$ induced by the $\Xi$-$n$-coalescent. Then, we
can express
\begin{equation}\label{FF-eq:onasblock}
O^{}_n=B^{(n)}_1(E^{}_n+T^{}_n).
\end{equation}
Based on this equation, we considered $O_n$ for
$\Lambda$-$n$-coalescents. For finite $n$, all moments $E(O_n^j)$,
$j\in\NN$, can again be computed recursively,
see~\cite[Thm.~4.1]{FF-Freund2019b}.  Since the recursion for fixed
$n$ is rather involved, yet relatively straightforward, let us focus
on the asymptotics of $O_n$ for $n\to\infty$.  Generally, for any
$\Xi$-$n$-coalescent, $E_n$ decreases in $n$ and
$E_n\stackrel{d}{\to}Exp (\mu_{-1})$ for $n\to\infty$, where
$\mu_{-1}=\int_{\Delta} \sum_{i=1}^{\infty}x_i
\frac{\Xi(\dd x)}{(x,x)}$.  Thus, for
$\Lambda$-$n$-coalescents without dust, $\mu_{-1}=\infty$ and
$E_n\to 0$ almost surely. In this case, $O_n$ is asymptotically 
$B^{(n)}(T')$, where $T'\stackrel{d}{=}Exp (\frac{\theta}{2})$, since
this is the time until the first mutation in the coalescent
tree on its path from leaf to root (which does not change with $n$,
since the $n$-coalescents are nested, branches are added if
$n$ grows).  $T'$ is independent from all $\Lambda$-$n$-coalescents.
Kingman's correspondence ensures that the asymptotic frequency
$\lim_{n\to\infty}B^{(n)}_1(T')/n=f_1(T')$ exists almost
surely, where $f_1(t)$ is the asymptotic frequency of the block
including 1 at time $t$.  This ensures $n^{-1}O_n\to f_1(T')$
 almost surely. When we additionally plug in the
representation of the moments of $f_1(t)$ which follows
from~\cite[Eq. (50) + Prop. 29]{FF-Pitman1999}, we can fully
characterise the distribution of $f_1(T')$ by its moments
\begin{equation}\label{FF-eq:onlimnodust}
  E\big[\big( f^{}_1(T')\big)^k\,\big]=1-\sum_{r=2}^{k+1} a^{}_{k+1,r}
  \frac{\frac{\theta}{2}}{\lambda^{}_r+\frac{\theta}{2}}, 
\end{equation}  
where $\lambda_r$ is the total rate of the $\Lambda$-coalescent in a
state with $r$ blocks and $a_{k,r}$ is a rational function of
$\lambda_2,\ldots,\lambda_k$, defined as
in~\cite[Prop. 29]{FF-Pitman1999}.  For the
Bolthausen--Sznitman-$n$-coalescent\index{coalescent!Bolthausen--Sznitman!$n$-} 
($\Lambda=Beta (1,1)$), we can show that the
distribution of $f_1(T')$ is a Beta distribution.  For any
$\Lambda$-$n$-coalescent, $f_1$ is a pure jump process and for
$\Lambda=Beta (1,1)$, the countably many jumps are happening at
i.i.d. $Exp(1)$ distributed times and their heights $(J_i)_{i\in\NN}$
are governed by a Poisson--Dirichlet distribution with parameters
$(0,1)$, see~\cite[Cor.~16, Prop.~30]{FF-Pitman1999}.  Thus, each jump
contributes to $f_1(T')$ if its jump time is smaller than $T'$, which
happens with probability $\left(1+\frac{\theta}{2}\right)^{-1}$.  This
shows that $f_1(T')\stackrel{d}{=}\sum_{i\in\NN}B_iJ_i$, where
$(B_i)_{i\in\NN}$ are i.i.d. Bernoulli random variables with success
probability $\left(1+\frac{\theta}{2}\right)^{-1}$, and this sum has a
$Beta\big(\left(1+\frac{\theta}{2}\right)^{-1},
\frac{\theta}{2}\left(1+\frac{\theta}{2}\right)^{-1}\big)$-distribution,
which can be seen from the construction of the Poisson--Dirichlet
distribution via a Poisson point process, as
e.g. in~\cite[Sect. 4.11]{FF-ABT}.

If the $\Lambda$-$n$-coalescent has dust, the asymptotic decoupling of
$f_1$ and the time until the first mutation on the path from
individual 1 to the root that is not on the external branch breaks
down. However, if a $\Xi$-coalescent has dust and stays infinite, the
block frequency $f_1$ can be characterised from jump to jump, which I
showed in joint work with M\"ohle.\index{coalescent!Xi@$\Xi$!asymptotic frequencies}
\begin{proposition}\label{FF-prop:f1dust}\cite[Thm. 1]{FF-freund2017size}
  In any $\Xi$-coalescent with dust that stays infinite,
  $(f_1(t))_{t\geq 0}$ is an increasing pure jump process with
  c\`adl\`ag paths, $f_1(0)=0$ and $\lim_{t\to\infty}f_1(t)=1$, but
  $f_1(t)<1$ for $t>0$ almost surely. The waiting times between the
  almost surely infinitely many jumps are distributed as independent
  ${\rm Exp}(\mu_{-1})$ random variables. Its jump chain
  $(f_1[k])_{k\in\NN}$ can be expressed via stick-breaking
\begin{equation}\label{FF-eq:stickbreak}
f^{}_1[k]=\sum_{i=1}^{k}X^{}_i\prod_{j=1}^{i-1}(1-X^{}_j),
\end{equation}
where the $(X_j)_{j\in\NN}$ are pairwise uncorrelated, $X_j>0$ almost
surely and $E(X_j)=\gamma:=\frac{\Xi(\Delta)}{\mu_{-1}}$ for all
$j\in\NN$. Moreover, $E(f_1[k])=1-(1-\gamma)^k$.  In general, the 
$(X_j)_{j\in\NN}$ are neither independent nor identically distributed.
\end{proposition}

\begin{proof}
  Here, only a sketch of the proof is excerpted
  from~\cite{FF-freund2017size}, focusing on the waiting times between
  jumps and their expected height.  The expected height is given in
  terms of the mass added to the asymptotic frequency $f_1(t-)$ if the
  $k$th jump is at time $t$, and thus has the form $X_k(1-f_1(t-))$.
  First, observe that if the $\Xi$-coalescent has dust,
  $\Xi(\{(0,0,\ldots)\})=0$, which means that any $\Xi$-$n$-coalescent
  can be constructed by colouring according to a Poisson point process
  on $[0,\infty)\times \Delta$ with intensity
  $\dd t\otimes \frac{\Xi(\dd x)}{(x,x)}$ without having to
  add additional mergers of two blocks.  Recall the Poisson point
  construction\index{paintbox construction} of a $\Xi$-$n$-coalescent
  from Section~\ref{FF-sect:intro}, especially how mergers are
  determined by throwing independent balls on a `paintbox'
  $x\in\Delta$ and merging blocks whose balls have the same colour
  (landed in the same compartment of $x$).  As also mentioned in the
  first section, if the $\Xi$-coalescent has dust, there are almost
  surely only finitely many Poisson points $(t,x)$ where the block
  containing $i$ is coloured. More precisely, each Poisson point is
  affecting the block containing 1 with probability
  $\sum_{i\in\NN}x_i$.  Dividing the points of the original Poisson
  point into those that colour the block including 1 leads to a marked
  Poisson process, so the Poisson points affecting the block including~1 
  form again a Poisson point process on $[0,\infty)\times\Delta$
  with intensity measure
  $\nu_1=\dd t\otimes\sum_{i\in\NN}x_i\frac{\Xi(\dd x)}{(x,x)}$.
  Since the coalescent stays infinite, every such point, due to the
  strong law of large numbers for exchangeable random variables, will
  colour infinitely many blocks with this colour.  For $f_1$, we can
  thus describe its jump by just merging the blocks of the
  $\Xi$-coalescent according to the colouring at each Poisson point
  affecting the block including 1 (though the Poisson points not
  affecting 1 do change the frequencies of these blocks). The
  probability for not having a Poisson point affecting 1 with time
  component $t$ in any interval of length $t_0$ is given by
  $\exp\big( -\nu_1([t',t'+t_0)\times \Delta)\big)=\exp(-t_0\mu_{-1})$,
  which verifies the distribution of the waiting times.  At such a Poisson
  point $(t,x)$, any other block merges with the block including~1 if
  it has the same colour $i$ (which the block including~1 picks with probability $\frac{x_i}{\sum_{i\in\NN}x_i}$), so with probability $x_i$,
  regardless of its frequency. Each $x$ is drawn from the probability
  distribution
  $(\mu_{-1})^{-1}\sum_{i\in\NN}x_i\frac{\Xi(\dd x)}{(x,x)}$.
  All blocks not containing~1 have total frequency $1-f_1(t-)$ when
  potentially merging at $(t,x)$, so the average fraction of mass from
  $1-f_1(t-)$ added to $f_1(t-)$
  is
  $$
  E(X^{}_k)= \frac{1}{\mu^{}_{-1}} \int_{\Delta}\sum_{i\in\NN}
  \frac{x_i^2}{\sum_{k\in\NN}x^{}_k}
    \sum_{\ell\in\NN} x^{}_\ell \frac{\Xi(\dd x)}{(x,x)} 
  =\frac{\Xi(\Delta)}{\mu^{}_{-1}}.
  $$
\end{proof}     
This result can now be applied to address the asymptotics of $O_n$ for
$\Lambda$-$n$-coalescents with dust if the coalescent stays infinite,
i.e. $\Lambda(\{1\})=0$. Let $T_1,T_2,\ldots$ be the i.i.d. jump times
of $f_1$. Clearly, $E_n=T_1$ for $n$ large enough.  Similarly as in
the case without dust, we can then see that there is a time $T'$, not
dependent on $n$, when the first mutation after time $T_1$ appears on
the path from leaf 1 to the root of the $\Xi$-$n$-coalescent, and
this, for $n$ large enough, will fall in between two jumps of $f_1$,
say $K$ and $K+1$.  Thus, Proposition~\ref{FF-prop:f1dust} ensures
that $\lim_{n\to\infty}n^{-1}O_n=f_1[K]$ exists almost surely. $K$ is
geometrically distributed on $\NN$ with success probability
$\frac{\mu_{-1}}{\mu_{-1}+\frac{\theta}{2}}$, since the exponential
distribution is memoryless and the probability that one exponential
random variable is smaller than another one independent of it is given
by the exponential rate of the first divided by the sum of the rates.
With Proposition~\ref{FF-prop:f1dust}, one can compute
$E(f_1[K])=1-\frac{\theta}{2\mu_{-1}}\frac{a}{1-a}$ with
$a=\left(1-\frac{\Lambda([0,1])}{\mu_{-1}}\right)
\left(\frac{\frac{\theta}{2}}{\frac{\theta}{2}+\mu_{-1}}\right)$.

The process $(f_1(t))_{t\geq 0}$ has some interesting further
properties. While it is Markovian when observed in the
Bolthausen--Sznitman $n$-coalescent,
see~\cite[Cor. 16]{FF-Pitman1999}, it is not in general. For instance,
by further exploiting the Poisson construction, the following
properties of $f_1$ for Dirac coalescents can be
derived\index{coalescent!Xi@$\Xi$!asymptotic frequencies}\index{coalescent!Dirac}.

\begin{proposition}\cite[essentially Prop. 2]{FF-freund2017size}
  Let $\Lambda=\delta_p$, $p\in[\frac{1}{2},1)$ or
  $p\in(0,\frac{1}{2})$ and transcendental, and $q:=1-p$. $f_1$ takes
  values in the set
\[
  \cM^{}_p:=\Big\{\sum_{i\in\NN} b^{}_i pq^{i-1} : b^{}_i\in\{0,1\},
  \sum_{i\in\NN} b^{}_i < \infty\Big\} .
\]
For $x=\sum_{i\in\NN}b_i pq^{i-1}\in\mathcal{M}_p$, we have
$$
P(f^{}_1[1]=x)=pq^{j-1} \prod_{i\in J\setminus \{j\}}P(Y+i\in J)
\prod_{i\in [j-1]\setminus J} P(Y+i\notin J)>0,
$$
where $Y\stackrel{d}{=}Geo (p)$, $J:=\{i\in\NN\mid b_i=1\}$ and $j:=\max J$.
The process $f_1$ is not Markovian
whereas its jump chain $(f_1[k])_{k\in\NN}$ is  Markovian.
\end{proposition}

Under the condition for $p$ from the proposition, knowing that
$f_1[1]=x$ allows to directly infer information about mergers at or
before the time the minimal clade is formed: Each collision of a Dirac
coalescent (on $\mathbb{N}$) merges a fraction of $p$ of all singleton
blocks present (the `dust'), so each each block has asymptotic
frequency from $\mathcal{M}_p$.  The condition on $p$ ensures that
$x\in\mathcal{M}_p$ has a unique representation in terms of the
$b_i$'s.  Moreover, the $b_i$'s encode at which Poisson point the
minimal clade appears and which fraction of singletons merged at
earlier Poisson points are merged to form the minimal clade.  If one
drops the condition on $p$, the distribution of $f_1[1]$ would become
more involved, since one would have to trace back which combinations
of $b_i$'s lead to the same $x$ and sum these.

\section*{Acknowledgements}
I thank Anton Wakolbinger, Ellen Baake and two anonymous referees for helpful comments and suggestions that improved the quality of this review.

\end{document}